\documentclass[12pt,reqno, a4paper]{amsart}
\usepackage{amssymb,amsmath}
\usepackage[final]{graphicx}
\usepackage{epic}
\usepackage{pstricks}
\usepackage{color}
\usepackage[ansinew]{inputenc}
\usepackage[english]{babel}
\usepackage{latexsym}
\usepackage{eufrak}
\usepackage{mathrsfs}
\usepackage{fancybox}
\usepackage{fancyhdr}
\usepackage{array}
\newtheorem{thm}{Theorem}
\newtheorem{cor}[thm]{Corollary}

\def\Z{\mathbb{Z}}
\def\C{\mathbb{C}}
\def\N{\mathbb{N}}

\title[linearization coefficients of Bessel polynomials]{An explicit formula for the linearization coefficients of Bessel polynomials}
\author{Mohamed Jalel Atia}
\address[Mohamed Jalel Atia]{Facult\'e de Science, Universit\'e de Gab\`es, Tunisie}
\email{jalel.atia@gmail.com}
\author{Jiang Zeng}
\address[Jiang Zeng]{Universit\'{e} de Lyon; Universit\'{e} Lyon 1; Institut Camille Jordan; UMR 5208 du CNRS; 43, boulevard du 11 novembre 1918, F-69622 Villeurbanne Cedex, France}
\email{zeng@math.univ-lyon1.fr}

\date{\today}
\begin{document}
\maketitle

\begin{abstract}
We prove a single sum  formula for  the linearization coefficients  of the Bessel polynomials. In two special cases we show that 
our formula reduces indeed to 
  Berg and Vignat's formulas  in their proof of the  positivity results about 
these  coefficients  (Constructive Approximation, {\bf 27} (2008), 15-32). As a bonus we also obtain a generalization of 
an integral formula of Boros and Moll (J. Comput. Appl. Math. 106 (1999), 361-368).
 \end{abstract}
{\bf A.M.S. Classification:} Primary  33C10;  Secondary 60E05

\section{Introduction}

The Bessel polynomials $q_n$ of degree $n$ are defined by
\begin{equation} \label{poly_bess}
q_n(u)=\sum_{k=0}^n{(-n)_k2^k\over (-2n)_kk!}u^k,
\end{equation}
where we use the Pochhammer symbol $(z)_n:=z(z+1)\ldots (z+n-1)$ for $z\in \C$ and $n\in \N$.
The first values are 
$$
q_0(u)=1,\quad q_1(u)=1+u, \quad q_2(u)=1+u+\frac{u^2}{3}.
$$
Using hypergeometric functions, we have $q_n(u)=\ _1F_1(-n;-2n;2u)$.
They are normalized according to $q_n(0)=1$, and thus differ from the monic normalization $\theta_n(u)$ in Grosswald's monograph~\cite{Gro}:
$$\theta_n(u)=\displaystyle{(2n)!\over n!2^n}q_n(u).$$
The polynomials $\theta_n$ are sometimes called the reverse Bessel polynomials and $y_n(u)=u^n\theta_n({1\over u})$ the ordinary Bessel polynomials.
These Bessel polynomials are, then, written as
\begin{equation}\label{yn}
y_n(u)=\displaystyle{(2n)!\over n!2^n}u^nq_n({1\over u})=\sum_{k=0}^n{(n+k)!\over 2^k k!(n-k)!}u^k.
\end{equation}
The so-called  linearization  coefficients $\beta_{k}^{(n,m)}$  of the Bessel polynomials~\cite{BV1} are defined by
\begin{equation} \label{eq:lin}
q_n(au)q_{m}((1-a)u)=\sum_{k=0}^{n+m}\beta_{k}^{(n,m)}(a)q_{k}(u).
\end{equation}
For example,  we have $\beta_k^{(3,5)}=0$ for $k=0,1,2$ and 
\begin{align*}
\beta_3^{(3,5)}(a)&={a}^{7},\\
\beta_4^{(3,5)}(a)&=7\,{a}^{7} \left( 1-a \right),\\
\beta_5^{(3,5)}(a)&=\left( 1-a \right) ^{2} \left( 84\,{a}^{6}-126\,{a}^{5}+126\,{a}^{4}-
84\,{a}^{3}+36\,{a}^{2}-9\,a+1 \right),\\
\beta_6^{(3,5)}(a)&={\frac {11}{5}}\, \left( 1-a \right) ^{3} \left( 126\,{a}^{4}-168\,{a
}^{3}+108\,{a}^{2}-36\,a+5 \right),\\
\beta_7^{(3,5)}(a)&={\frac {143}{5}}\,{a}^{2} \left( 1-a \right) ^{4}\left( 12\,{a}^{2}-9\,a+2
 \right), \\
\beta_8^{(3,5)}(a)&=143\,{a}^{3} \left( 1-a \right) ^{5}.
\end{align*}
Recently Berg and Vignat~\cite{BV1, BV2} have established   some important results about 
the coefficients $\beta_{k}^{(n,m)}(a)$ with applications. 
In particular,  they proved \cite[Theorem 2.1]{BV1} that $\beta_n^{(n,0)}(a)=a^n$  and, 
for $0\leq k\leq n-1$,
\begin{equation} \label{cnk}
\beta_k^{(n,0)}(a)=a^k(1-a)\frac{{n\choose k}}{{2n\choose 2k}}\sum_{j=0}^{(n-j-1)\wedge k}{n+1\choose k-j}{n-k-1\choose j}(1-a)^j.
\end{equation}
Moreover,
for $n,m\geq 1$, Berg and Vignat~\cite{BV1}  have  proved that  the coefficients $\beta_{k}^{(n,m)}(a)$ satisfy 
the following recurrence relation~\cite[Lemma 3.6]{BV1}:
\begin{equation}\label{relationbeta}
\frac{1}{2k+1}\beta^{(n,m)}_{k+1}(a)=\frac{a^2}{2n-1}\beta^{(n-1,m)}_{k}(a)+\frac{(1-a)^2}{2m-1}\beta^{(n,m-1)}_{k}(a),
\end{equation}
for $k=0,1,...,m+n-1$ and $\beta_0^{(n,m)}(a)=0$. From \eqref{cnk} and \eqref{relationbeta} they 
derive the positivity of $\beta^{(n,m)}_{k}(a)$  when  $0\leq a\leq 1$  and
also  that $\beta^{(n,m)}_{k}(a)=0$ for $k<\min(m,n)$. However, an explicit   single sum formula 
for $ \beta_{k}^{(n,m)}(a)$  is missing in their paper.

Our main result is an explicit single sum formula for $  \beta_{i}^{(n,m)}(a)$, which provides actually  the unique solution of  the recurrence system~\eqref{relationbeta} with the boundary condition \eqref{cnk}.

\begin{thm}\label{thm:A-Z}
For $k=0,1,...,n+m$, we have 
\begin{align}\label{eq:thm}
  \beta_{k}^{(n,m)}(a)
  &=a^{2n+m-k}(1-a)^{-n+k}\frac{(1/2)_{n+m-k}(1/2)_{k}}{(1/2)_n(1/2)_{m}}
 \\
&\times \sum_{j=0}^{2(n+m-k)}(-1)^j{n+m+1\choose 2n+2m-2k-j}{-m+k+j\choose j}a^{-j}. \nonumber
\end{align}
Equivalently,   the coefficients  $  \beta_{k}^{(n,m)}(a)$ can be written  in terms of hypergeometric function:
\begin{itemize}
\item[(i)]  if  $k\geq \left\lceil (n+m-1)/2\right\rceil$, then
\begin{align}\label{eq1:thm2F1}
  \beta_{k}^{(n,m)}(a)&= a^{2n+m-k}(1-a)^{-n+k}\frac{(1/2)_{n+m-k}(1/2)_{k}}{(1/2)_n(1/2)_{m}}{n+m+1\choose 2n+2m-2k}\nonumber \\
&\times 
      {}_2F_1\left(\begin{array}{cc}
-2m-2n+2k,&-m+k+1\\
&\hspace{-3cm}2k-m-n+2\end{array}; \frac{1}{a}
\right);
\end{align}
\item[(ii)] if  $k\leq\left\lfloor(n+m-1)/2\right\rfloor$, then
\begin{align}\label{eq2:thm2F1}
  \beta_{k}^{(n,m)}(a)&= a^{n+k+1}(1-a)^{-n+k}\frac{(1/2)_{n+m-k}(1/2)_{k}}{(1/2)_n(1/2)_{m}}{n-k-1\choose n+m-2k-1}\nonumber \\
&\times (-1)^{n+m-1}
      {}_2F_1\left(\begin{array}{cc}
-m-n-1,&n-k\\
&\hspace{-2cm}n+m-2k\end{array}; \frac{1}{a}
\right).
\end{align}
\end{itemize}
\end{thm}

We first derive   some immediate consequences from the above theorem.  Assume that $n\leq m$ and $0\leq a\leq 1$.
\begin{enumerate}
\item  If $k<n$, then the binomial coefficient ${n-k-1\choose n+m-2k-1}$ vanishes because $n+m-2k-1>n-k-1$. Thus
formula  \eqref{eq2:thm2F1} implies 
that $ \beta_{k}^{(n,m)}(a)=0$ for $k<m\wedge n$. 
\item  If $k\geq n$, then each term in \eqref{eq:thm}  is a polynomial in $a$. 
Indeed,  it is clear  that $(1-a)^{-n+k}\in \Z[a]$, and 
\begin{itemize}
\item if $k\geq m$, then  $2n+m-k-j\geq 2n+m-k-2(n+m-k)=k-m\geq 0$ for $0\leq j\leq 2(n+m-k)$, 
\item if $k<m$, then 
the binomial coefficient ${-m+k+j\choose j}$ does not vanish only if $-m+k+j<0$, that is, $j<m-k$, therefore, 
$2n+m-k-j> 2n+m-k-(m-i)=2n\geq 0$.
\end{itemize}
\item  From  definition \eqref{eq:lin} we derive immediately the symmetry property:
\begin{equation}\label{eq:sym}
\beta_{k}^{(n,m)}(a)=\beta_{k}^{(m,n)}(1-a).
\end{equation}
This identity  follows also from   \eqref{eq1:thm2F1} (resp.  \eqref{eq2:thm2F1}) and  Pfaff's transformation (see \cite[p. 94]{aar}):
\begin{align}\label{eq:pfaff}
{}_2F_1\left(\begin{array}{cc}
A&B\\
&\hspace{-20pt}C\end{array}; x
\right)=(1-x)^{-A}   {}_2F_1\left(\begin{array}{cc}
A,&C-B\\
&\hspace{-20pt}C\end{array}; \frac{x}{x-1}
\right),
\end{align}
with $A=-2n-2m+2k$, $B=-m+k+1$, $C=-n-m+2k+2$ and $x=1/a$  if  $k\geq \left\lceil (n+m-1)/2\right\rceil$
(resp.  $A=-n-m-1$, $B=n-k$, $C=n+m-2k$ and $x=1/a$ if $k\leq\left\lfloor(n+m-1)/2\right\rfloor$).
\end{enumerate}

We  shall prove  Theorem~1 in Section~2.  As applications of our Theorem~1, we shall 
 derive a formula of Berg and Vignat~\cite{BV1}  when  $m=n$, a positivity result as well as  a formula 
 generalizing an integral  evaluation of Boros and Moll~\cite{BM} in Section~3.
\section{Proof of Theorem 1}
We shall first derive   a double sum 
formula for the coefficients  $\beta_i^{(n,m)}(a)$.  
\subsection{A double sum}
Assume that $m\geq n$.   Let
\begin{equation} \label{coeff-poly_bess}
\alpha_k^{(n)}={(-n)_k2^k\over (-2n)_kk!}.
\end{equation}
By definition \eqref{poly_bess}  we have
\begin{align}\label{eq:mul}
q_n(au)q_m((1-a)u)
&=\sum_{k=0}^{m+n} u^k \sum_{j=0}^k \alpha_j^{(n)}  \alpha_{k-j}^{(m)} a^j(1-a)^{k-j}.
\end{align}
Recall the connection formula due to   Carlitz \cite{ca}:
\begin{equation}\label{eq:carlitz}
u^k=\sum_{i=0}^k \delta_i^{(k)} q_i(u), \quad k=0,1, \ldots,
\end{equation}
where 
$$
\delta_i^{(k)} =\left\lbrace \begin{array}{cc}
\frac{(k+1)!}{2^k}\frac{(-1)^{k-i} (2i)!}{(k-i)!i!(2i+1-k)!}&\quad \textrm{for}\quad  \frac{k-1}{2}\leq i\leq k,\\
0&\quad \textrm{for}\quad  0\leq i< \frac{k-1}{2}.
\end{array}\right.
$$
Plugging  \eqref{eq:carlitz}  into \eqref{eq:mul} and extracting the coefficient of $q_i(u)$ we obtain
\begin{align}\label{eq:key}
\beta_i^{(n,m)}(a)&=\sum_{k=i}^{m+n}\delta_i^{(k)}\sum_{j=0}^{ k}  \alpha_j^{(n)}  \alpha_{k-j}^{(m)} a^j(1-a)^{k-j}.
\end{align}
Replacing $k$ by $k+i$ and $j$ by $n-j$ we obtain
\begin{align}\label{eq:key}
\beta_i^{(n,m)}(a)
&=\sum_{k=0}^{i+1}\delta_i^{(k+i)}\sum_{j=0}^n \alpha_{n-j}^{(n)}  \alpha_{k+i-n+j}^{(m)} a^{n-j}(1-a)^{k+i-n+j}\nonumber\\
&=
\frac{(2i)!}{i!}(1-a)^{i-n}
\sum_{k=0}^{i+1}\sum_{j=0}^n \frac{(-1)^{k}(k+i+1)!}{k!(i+1-k)!}\\
&\hspace{1cm}\times \frac{(-n)_{n-j}(-m)_{k+i-n+j}a^{n-j} (1-a)^{k+j}}{(-2n)_{n-j} (n-j)!(-2m)_{k+i-n+j}(k+i-n+j)!}.\nonumber
\end{align}
As $\beta_i^{(n,m)}(a)=0$ if $i<n$,  replacing $i$ by $n+m-i$ and assuming $i\leq m$ we obtain, after some simplification,
\begin{align}\label{eq:key1}
\beta_{n+m-i}^{(n,m)}(a)
&=4^{-i} \frac{(1/2)_{n+m-i}}{(1/2)_n (1/2)_m}a^{n+i}(1-a)^{m-i}\frac{(m+i)!}{(m-i)!i!}\nonumber \\
&\times \sum_{k\geq 0}(-1)^k \frac{(m+n-i+2)_k(-n-m+i-1)_k}{k!}\nonumber\\
&\times \sum_{j\geq 0} \frac{(-n)_{j}(-i)_{k+j} (n+1)_j a^{-i+k} (1-1/a)^{k+j}}{j!(-m-i)_{k+j} 
(1+m-i)_{k+j} }.
\end{align}

By binomial formula, we have 
$$a^{-i+k}(1-1/a)^{k+j}=\sum_{t\geq 0} \frac{(-k-j)_t}{t!}a^{-i+k-t}.
$$
Let $l=k+j$ and $r=i+t-k$ we can write
\begin{align}\label{eq:key1}
\beta_{n+m-i}^{(n,m)}(a)
&= \frac{(1/2)_{n+m-i}(1/2)_i}{(1/2)_n (1/2)_m}a^{n+i}(1-a)^{m-i}{m+i\choose 2i}\nonumber \\
&\times \sum_{r=0}^{2i}a^{-r} \sum_{l\geq 0} \frac{(-i)_l(-n)_l (n+1)_l}{(-m-i)_l(1+m-i)_ll!}\\
&\times
\sum_{k\geq 0}\frac{(-l)_k(-l)_{r-i+k}  (m+n-i+2)_k (-n-m+i-1)_k}{k!(r-i+k)!(n-l+1)_k(-n-l)_k}.\nonumber
\end{align}
For $i=0,1$ and  2,  the above formula yields immediately the following explicit expressions:
\begin{align*}
\beta_{n+m}^{(n,m)}(a)
&=a^n(1-a)^{m}\frac{(1/2)_{n+m}}{(1/2)_n(1/2)_{m}},
\end{align*}
\begin{align*}
\beta_{n+m-1}^{(n,m)}(a)
&=a^{n+1}(1-a)^{m-1} \frac{(1/2) (1/2)_{n+m-1}}{(1/2)_n (1/2)_m} \\
&\qquad\times \left( {n+m+1\choose 2}   -{n+m+1\choose 1}{n\choose 1}a^{-1}+{n+1\choose 2}a^{-2}\right),
\end{align*}
and 
\begin{align*}
\beta_{n+m-2}^{(n,m)}(a)
&=a^{n+2}(1-a)^{m-2} \frac{(1/2)_2 (1/2)_{n+m-2}}{(1/2)_n (1/2)_m} \\
& \times \left({m+n+1\choose 4}-{n+m+1\choose 3}{n-1\choose 1}a^{-1}\right.\\
&+{n+m+1\choose 2}{n\choose 2}a^{-2}-{n+m+1\choose 1}{n+1\choose 3}a^{-3}\\
&\left.+{n+2\choose 4} a^{-4}\right).
\end{align*}
Further computation  led us to conjecture  the formula \eqref{eq:thm}, which is equivalent to \eqref{eq1:thm2F1} and \eqref{eq2:thm2F1}.

Our proof  consists of verifying that the conjectured formula \eqref{eq:thm} does satisfy
 the boundary condition \eqref{cnk} and the recurrence relation~\eqref{relationbeta}.

\subsection{The boundary condition}
We first show that the formula \eqref{eq:thm} reduces to \eqref{cnk} when $m=0$.
Let 
$c_k^{(n)}(a)=\beta_k^{(n,0)}(a)$ 
for $0\leq k\leq n$.  Clearly we have  $c_n^{(n)}(a)=a^n$ from \eqref{eq:thm}. 
Assume that $0\leq k\leq n-1$,  using the symmetry  $c_n^{(n)}(a)=\beta_k^{(0,n)}(1-a)$, 
we derive from   \eqref{eq:thm} that
\begin{align*}
c_k^{(n)}(a)&=(1-a)a^{k}(1/2)_{n-k}
\frac{(1/2)_{k}}{(1/2)_{n}}\nonumber \\
&\times \sum_{j=0}^{2n-2k}(-1)^j{n+1\choose 2n-2k-j}{k-n+j\choose j}(1-a)^{n-k-j-1}.
\end{align*}
As  ${k-n+j\choose j}=0$ if $j\geq n-k$, replacing $j$ by $n-k-1-j$ we obtain
\begin{align*}
c_k^{(n)}(a)&=(1-a)a^{k}(1/2)_{n-k}
\frac{(1/2)_{k}}{(1/2)_{n}}\nonumber \\
&\times \sum_{j=0}^{(n-k-1)\wedge k}(-1)^{n-k-1-j}{n+1\choose n-k+1+j}{-1-j\choose n-k-1-j}(1-a)^{j}.
\end{align*}
Since
$$
(1/2)_{n-k}
\frac{(1/2)_{k}}{(1/2)_{n}}=\frac{{n\choose k}}{{2n\choose 2k}},
\quad
{n+1\choose n-k+1+j}={n+1\choose k-j},
$$
and 
\begin{align*}
(-1)^{n-k-1-j}{-1-j\choose n-k-1-j}&=\frac{(j+1)(j+2)\ldots (n-k-1)}{(n-k-1-j)!}\\
&={n-k-1\choose j},
\end{align*}
we derive the desired formula \eqref{cnk}. \qed
\medskip

\subsection{The recurrence relation}
It remains to  verify   \eqref{relationbeta} for the conjectured formula~\eqref{eq:thm}.
One of  Gauss' relation for contiguous hypergeometric
functions~\cite[p. 71]{Ran}  reads:
\begin{align}\label{eq:gauss}
&C \, {}_2F_1\left(\begin{array}{cc}
A,&B\\
&\hspace{-1cm} C
\end{array}; z\right)\\
&=(C-A)z{}_2F_1\left(\begin{array}{cc}
A,&B+1\\
&\hspace{-1cm} C+1
\end{array}; z\right)+C(1-z){}_2F_1\left(\begin{array}{cc}
A,&B+1\\
&\hspace{-1cm} C
\end{array}; z\right).\nonumber
\end{align}
This can also be verified straightforwardly by comparing the coefficients of $z^k$ for each $k\in \N$. Similarly we have the following
 contiguous  relation:
\begin{align}\label{eq:contiguous}
&B(1-z) \, {}_2F_1\left(\begin{array}{cc}
A,&B+1\\
&\hspace{-1cm} C+1
\end{array}; z\right)\\
&=(B-C)z{}_2F_1\left(\begin{array}{cc}
A,&B\\
&\hspace{-1cm} C+1
\end{array}; z\right)
+C{}_2F_1\left(\begin{array}{cc}
A-1,&B\\
&\hspace{-1cm} C
\end{array}; z\right).\nonumber
\end{align}
Now,  the recurrence relation \eqref{relationbeta} is satisfied by  the conjectured formula \eqref{eq:thm} since 
\begin{itemize}
\item for  $k\geq \left\lceil (n+m-1)/2\right\rceil$,  this is equivalent to \eqref{eq:gauss} with
 $z=1/a$ and 
$
A=-2m-2n+2k+2, \quad B= -m+k+1, \quad C=-n-m+3+2k.
$
 \item for $k\leq\left\lfloor(n+m-1)/2\right\rfloor$,   this is equivalent to \eqref{eq:contiguous} with
 $z=1/a$ and 
$
A=-m-n, \quad B= -n-k-1, \quad C=n-m-2k.
$
\end{itemize} 
 The formula is thus proved. \qed

\bigskip

As a byproduct, comparing   \eqref{eq:thm} and \eqref{eq:key1} we derive  immediately  the following double sum identity.
\begin{cor}
For all positive  integers $i, m, n$ and $r=0, 1, \ldots, 2i$ we have 
\begin{align}\label{eq:cor}
{m+i\choose 2i}&\sum_{l=\max\{0, i-r\}}^{\min\{i,n\}} \sum_{k=\max\{0, i-r\}}^{\min\{l, l+i-r\}}\frac{(-i)_l(-n)_l (n+1)_l}{(-m-i)_l(1+m-i)_ll!}\nonumber\\
&\times
\frac{(-l)_k(-l)_{r-i+k} (m+n-i+2)_k (-n-m+i-1)_k}{k!(r-i+k)!(n-l+1)_k(-n-l)_k}\nonumber\\
&=(-1)^r{m+n+1\choose 2i-r}{n-i+r\choose r}.
\end{align}
\end{cor}

It would be interesting to find a direct  standard proof of \eqref{eq:cor} using the 
 hypergeometric function theory.
\section{Applications}
\subsection{A positivity result}
Theorem~1 implies  immediately the following positivity result.
\begin{cor} Suppose that $n<m$.  Then 
 the coefficient $ \beta_{k}^{(n,m)}(a)$ is positive for  $n\leq k<m$. 
\end{cor}
\begin{proof}
For  $(n+m-1)/2\leq k<m$,  we have  $-m+k+1\leq 0$, $-2m-2n+2k<0$ and $2k-m-n+2>0$. 
Hence each summand in  \eqref{eq1:thm2F1} is positive and so is  $\beta_k^{(n,m)}(a)$.
For  $n\leq k\leq (n+m-1)/2$, we derive 
 from  \eqref{eq2:thm2F1}   and  \eqref{eq:pfaff}  that 
\begin{align}
\label{eq2:positivity}
  \beta_{k}^{(n,m)}(a)
  &= a^{k-m}(1-a)^{m+k+1}\frac{(1/2)_{n+m-k}(1/2)_{k}}{(1/2)_n(1/2)_{m}}\\
  &\times {n-k-1\choose n+m-2k-1}
 {}_2F_1\left(\begin{array}{cc}
-m-n-1,&m-k\\
&\hspace{-2cm}n+m-2k\end{array}; \frac{1}{1-a}
\right).\nonumber
\end{align}
In this case,  we have $-m-n-1<0$, $m-k>0$ and $m+n-2k>0$,  therefore  each summand in  \eqref{eq2:positivity} is positive and so
is  $ \beta_{k}^{(n,m)}(a)$.
Summarizing the above two cases,  we obtain the desired positivity.
\end{proof}

 \subsection{A formula of Berg and Vignat}
Let $\beta_i^{(n)}(a)$ be the coefficients defined by
\begin{equation} \label{qnauqn1mau}
q_n(au)q_n((1-a)u)=\sum_{i=0}^{n}\beta_i^{(n)}(a)q_{n+i}(u).
\end{equation}
\begin{cor}  For $0\leq i\leq n$ we have
\begin{align}\label{betanak}
\beta_i^{(n)}(a)&=\frac{(4a(1-a))^i}{4^{n}}\frac{(-n)_i(n+\frac{1}{2})_i}{i! (-n+\frac{1}{2})_i}\\
&\times{}_2F_1\left(\begin{array}{cc}
-n+i,&-n-\frac{1}{2}\\
&\hspace{-40pt}\frac{1}{2}\end{array}; (2a-1)^2
\right).\nonumber
\end{align}
\end{cor}
\begin{proof} By definition we have $\beta_i^{(n)}(a)=\beta_{n+i}^{(n,n)}(a)$.  It follows from \eqref{eq:thm}
that
\begin{align*}
\beta_i^{(n)}(a)
&=a^{2n-i}(1-a)^{i}(1/2)_{n-i}
\frac{(n+1/2)_{i}}{(1/2)_{n}}{2n+1\choose 2n-2i} \\
&\qquad\times {}_2F_1\left(\begin{array}{cc}
-2n+2i,&i+1\\
&\hspace{-40pt}2i+2\end{array}; \frac{1}{a}
\right).
\end{align*}
Applying the quadratic transformation formula \cite[p. 127]{aar} :
 \begin{align*}\label{eq:AAR}
 {}_2F_1\left(\begin{array}{cc}
a,&b\\
&\hspace{-20pt}2a
\end{array}; x
\right)=(1-x/2)^{-b} {}_2F_1\left(\begin{array}{cc}
b/2,& (b+1)/2\\
&\hspace{-20pt} a+1/2\end{array}; \left(\frac{x}{2-x}\right)^2
\right),
 \end{align*}
  we obtain
\begin{align*}
\beta_i^{(n)}(a)&=a^{2n-i}(1-a)^{i}(1/2)_{n-i}
\frac{(n+1/2)_{i}}{(1/2)_{n}}{2n+1\choose 2n-2i}\\
&\qquad\times   \left(  \frac{2a-1}{2a}\right)^{2n-2i}{}_2F_1\left(\begin{array}{cc}
i-n,&i-n-1/2\\
&\hspace{-40pt} i+3/2\end{array}; (2a-1)^{-2}
\right).
\end{align*}
Reversing the order of summation in the last ${}_2F_1$ and 
using $(x)_{n-k}=(-1)^k\frac{(x)_n}{(-x-n+1)_k}$ we 
recover \eqref{betanak}.
\end{proof}

\subsection{An integral evaluation} 
Recall \cite{BV1} the Student $t$-distribution with parameter $\nu>0$:
$$
f_\nu(x)={A_\nu\over (1+x^2)^{\nu+{1\over 2}}},\quad A_\nu={\Gamma(\nu+{1\over 2})\over
\Gamma({1\over 2})\Gamma(\nu)}.
$$
For $\nu=n+\frac{1}{2}$ with $n\in \N$ we say that 
the distribution has $2n+1$  degree of freedom. 
The Fourier transform  of the density $f_\nu$ is an even function, which for $\xi\geq 0$, can be expressed as
$$
k_\nu(\xi)=\int_{-\infty}^\infty f_\nu(x)e^{-ix\xi}dx=\frac{2^{1-\nu}}{\Gamma(\nu)}\xi^\nu K_v(\xi),
$$
where $K_\nu$ is the modified Bessel function of the third kind, also called the Macdonald function. 

For $\nu=n+\frac{1}{2}$ it is known that  $k_\nu(u)=e^{-u} q_n(u)$ for $u\geq 0$.
Thus  equation~\eqref{eq:lin} is equivalent to
$${1\over a}f_{n+{1\over 2}}({x\over a})*{1\over 1-a}f_{m+{1\over 2}}({x\over 1-a})
=\sum_{k=n\wedge m}^{n+m} \beta_k^{(n,m)}(a)f_{k+{1\over 2}}(x),$$
where $0<a<1$  and $*$ is the ordinary convolution of densities.

For $a={1\over 2}$, replacing  $x$ by ${x\over 2}$ and multiplying  by ${1\over 2}$ on both sides, we obtain
$$f_{n+{1\over 2}}*f_{m+{1\over 2}}(x)
=\sum_{k=n\wedge m}^{n+m} {1\over 2}\beta_k^{(n,m)}({1\over 2})f_{k+{1\over 2}}({x\over 2}),
$$
which  is equivalent to the following   integral evaluation.
\begin{align*}
\int_{-\infty}^{+\infty}\frac{dy}{ (1+y^2)^{n+1}(1+(x-y)^2)^{m+1}}=
\frac{1}{2}\sum_{k=m\wedge n}^{m+n}
 \frac{A_{k+\frac{1}{2}}}{A_{n+\frac{1}{2}}A_{m+\frac{1}{2}}}
\frac{\beta_{k}^{(n,m)}(\frac{1}{ 2})} {(1+\frac{1}{4}x^2)^{k+1}}.
\end{align*}
By Theorem~1 with $a=1/2$,  the above formula can be simplified as follows.
\begin{cor} For positive integers $m$ and $n$ we have 
\begin{align}
\int_{-\infty}^{+\infty}\frac{n!m! 2^{n+m} \;dy}{(1+y^2)^{n+1}(1+(x-y)^2)^{m+1}}=
\frac{\pi}{2}\sum_{k=m\wedge n}^{m+n} 
\frac{\gamma_{k}^{(n,m)}}{(1+\frac{1}{4}x^2)^{k+1}},
\end{align}
where 
$$
\gamma_{k}^{(n,m)}=\sum_{j=0}^{2(n+m-i)}(-2)^j{n+m+1\choose 2n+2m-2i-j}{-m+i+j\choose j}.
$$
\end{cor}
When  $m=n$, as shown by Berg and Vignat~\cite{BV2}, 
the above formula  yields a new proof  of an integral  formula  of Boros and Moll~\cite{BM}.
\medskip

\medskip
{\noindent\bf Acknowledgement. } This work was done mainly during the first author's visit to Universit\'e Lyon 1 in spring  2011
and supported by a scholarship 
{``}S\'ejours Scientifiques de Haut Niveau (SSHN)" de l'Institut fran\c cais de Tunisie.

\end{document}